\documentclass[12pt]{amsart}

\usepackage[english]{babel}
\usepackage[utf8]{inputenc}
\usepackage{times}
\usepackage[T1]{fontenc}
\usepackage{graphicx}
\usepackage{amscd,amsmath,amsfonts,amstext,amssymb,amsbsy,amsopn,amsthm,eucal}
\usepackage{txfonts}
\usepackage[figuresright]{rotating}
\usepackage{longtable,booktabs}
\usepackage{rotating,booktabs}
\usepackage{caption}
\usepackage{multirow}
\usepackage{tensor}
\usepackage[table]{xcolor}
\usepackage{fancyhdr}
\usepackage{color} 
\usepackage{hyperref}

\hypersetup{
  pdftitle={Talenti-Dirichlet},
  pdfauthor={CCL},
  pdfsubject={Eigenvalue},
  pdfkeywords={Faber-Krahn, Isoperimetric inequality, Dirichlet Laplacian, First eigenvalue}
  pdfpagelayout=SinglePage,
  pdfpagemode=UseOutlines,  pdfpagemode=UseOutlines,
  colorlinks,
  linkcolor=[rgb]{0,0,0.7},
  urlcolor=[rgb]{0,0,0.4},
  citecolor=[rgb]{0.4,0.1,0}
  }

\newcommand{\R}{\mathbb{R}}

\newcommand{\norm}[1]{\left\|#1\right\|}

\newcommand{\abs}[1]{\left|#1\right|}

\newcommand{\n}{\nabla}
\newcommand{\AVR}{\operatorname{AVR}(g)}

\newcommand{\gz}{g_{0}}

\newcommand{\ld}{\lambda}

\newtheorem{prop}{Proposition}[section]
\newtheorem{thm}[prop]{Theorem}

\newtheorem{lem}[prop]{Lemma}
\newtheorem{rem}[prop]{Remark}
\newtheorem{cor}[prop]{Corollary}
\newtheorem{defn}[prop]{Definition}

\numberwithin{equation}{section}

\begin{document}


\baselineskip=17pt


\title[Talenti Comparison]{Talenti's comparison theorem for Poisson equation and applications  on Riemannian manifold with nonnegative Ricci curvature}

\author[Daguang Chen,  Haizhong Li]{Daguang Chen, Haizhong Li}

\thanks{The  authors were supported by NSFC-FWO 11961131001 and  NSFC grant No. 11831005.}


\maketitle


\begin{abstract}
In this article, we prove Talenti's comparison theorem for Poisson equation on complete noncompact Riemannian manifold with nonnegative Ricci curvature. Furthermore, we obtain the Faber-Krahn inequality for the first eigenvalue of Dirichlet Laplacian, $L^1$- and $L^\infty$-moment spectrum, especially Saint-Venant theorem for torsional rigidity and a reverse H\"older inequality for eigenfunctions of Dirichlet Laplacian.
\end{abstract}

\section{Introduction}

Let $\Omega \subset \R^n$ be a bounded domain in Euclidean space, $f\in L^2(\Omega)$  be nonnegative  and $u$ be the solution to 
\begin{equation*}
\begin{cases}
-\Delta u= f,  & \text{in}\quad\Omega,\\
u =0, & \text{on}\quad \partial\Omega,
\end{cases}
\end{equation*}
and $v$ be the solution to 
\begin{equation*}
\begin{cases}
-\Delta v= f^\sharp,  & \text{in}\quad\Omega^\sharp,\\
v =0, & \text{on}\quad \partial\Omega^\sharp,
\end{cases}
\end{equation*}
where $\Omega^\sharp$ denotes the ball centered at the origin  satisfying $\abs{\Omega^{\sharp}}=\abs{\Omega}$
and $f^\sharp$ is Schwarz rearrangement of $f$.  In 1976,  Talenti \cite{Talenti76}  proved 
\begin{equation*}
u^\sharp(x)\leq v(x), \qquad \text{for } \, x\in \Omega^\sharp.
\end{equation*}
Talenti's comparison results were generalized to semilinear and nonlinear elliptic equations, for instance, in \cite{ALT90, Talenti79}, and parabolic equation,  for instance, in \cite{Bandle76,ALT90}. In 2018, Colladay et al. \cite{ColladayLM18}  generalized Talenti's result to compact Riemannian manifolds whose  Ricci curvature has positive lower bound. Recently,  Talenti's comparison results were extended to  $\operatorname{RCD(K,N)}$ spaces in \cite{Mondin0Vedovato}. We also refer the reader to excellent books \cite{Kawohl,Kesavan2006, Baernstein19} for related topics.

The main aim in the present work  is to generalize Talenti's comparison result  to  complete noncompact Riemannian manifold with nonnegative Ricci curvature. Let $(M,g)$ be a  noncompact, complete $n$ $(n\ge 2)$ dimensional  Riemannian manifold with nonnegative Ricci curvature and  positive asymptotic volume ratio, i.e. 
$$
{\sf AVR}(g)=\lim_{r\to \infty}\frac{\abs{B_x(r)}}{\omega_n r^n}>0.
$$ 
where $B_x(r)$ stands for open metric ball centered at $x\in M$ with radius $r>0$,  $\omega_n$  denotes by  the volume of the unit ball in Euclidean space $\R^n$.
Let   $\Omega $ be an open, bounded domain of $(M,g)$ with smooth boundary $\partial \Omega$. For a given nonnegative ( not identically zero)  $f\in L^2(\Omega)$,  we consider the following  problem
\begin{equation}\label{Equ:Dirichlet1}
\left\{
\begin{array}{ll}
-\Delta_g u= f,  & \mbox{in $\Omega$},\\
u =0, & \mbox{on $\partial\Omega$.}
\end{array}
\right.
\end{equation}
We will establish a comparison principle with the solution to the following problem
\begin{equation}\label{Equ:Dirichlet2}
\left\{
\begin{array}{ll}
-\Delta v= f^\sharp, & \mbox{in $\Omega^\sharp$}\\
v =0, & \mbox{on $\partial\Omega^\sharp$,}
\end{array}
\right.
\end{equation}
where $\Omega^\sharp$ denotes Euclidean ball centered at the origin  satisfying $\AVR \abs{\Omega^{\sharp}}=\abs{\Omega}$
 and $f^\sharp$ is the  Schwarz rearrangement of $f$.  Based on the isoperimetric inequality in \cite{Brendle}, we have
 \begin{thm}[]\label{th_main_f} Let $(M,g)$ be a  noncompact, complete $n$-dimensional Riemannian manifold with  nonnegative Ricci curvature and  Euclidean volume growth, i.e. $\AVR>0$. Assume that $\Omega$ is a bounded domain in $M$, $f\in L^2(\Omega)$ is nonnegative and $u$  is the weak solution to Problem \eqref{Equ:Dirichlet1}.  Let $\Omega^\sharp$ be an Euclidean ball satisfying $\AVR \abs{\Omega^{\sharp}}=\abs{\Omega}$ and $v$ be the solution to Problem \eqref{Equ:Dirichlet2}. Then we have
\begin{equation}\label{compare:uv}
u^\sharp(x)\leq v(x), \qquad x\in \Omega^\sharp.
\end{equation}
Moreover, the equality  holds in \eqref{compare:uv} if and only if  $(M,g)$ is isometric to Euclidean space $(\R^n,g_0)$ and $\Omega$ is isoperimetric to Euclidean ball $\Omega^\sharp$, where $g_0$ is the canonical metric of Euclidean space.
\end{thm}

As an application of Theorem \ref{th_main_f}, we obtain the Faber-Krahn inequality for the first eigenvalue of Dirichlet Laplacian.
\begin{cor}\label{cor:FK}
Let $(M,g)$ be a  noncompact, complete $n$-dimensional Riemannian manifold with  nonnegative Ricci curvature and  $\AVR>0$. Assume that $\Omega$ is a bounded domain in $M$, $\Omega^\sharp$ be an Euclidean ball satisfying $\AVR \abs{\Omega^{\sharp}}=\abs{\Omega}$  and $\lambda_1(\Omega)$ denotes the first eigenvalue of  Dirichlet Laplacian
\begin{equation}\label{Equ:DL}
\left\{
\begin{array}{ll}
-\Delta_g  u= \lambda_1(\Omega) \, u,  & \mbox{in $\Omega$},\\
u =0, & \mbox{on $\partial\Omega$.}
\end{array}
\right.
\end{equation}	
 Then we have
 \begin{equation}
 \label{ineq:FK}
 \lambda_1(\Omega^{\sharp})\leq \lambda_1(\Omega).
 \end{equation}
 where $\lambda_1(\Omega^\sharp)$ is the first eigenvalue of $\Omega^{\sharp}$.
Moreover, the equality  holds if and only if $(M,g)$ is isometric to Euclidean space $(\R^n,g_0)$ and $\Omega$ is isometric to Euclidean ball $\Omega^\sharp$.
\end{cor}
\begin{rem} The Faber-Krahn inequality \eqref{ineq:FK} has been  proved in \cite{FM202012} for $3\leq n\leq 7$ and in \cite{BK202012} for all dimension $n$, both in terms of   P\'{o}lya-Szeg\"o principle. However, the proof we given here differs from theirs and the equality case follows Talenti's comparison result  in Theorem \ref{th_main_f} directly. The inequality \eqref{ineq:FK}   can also be written as
	\begin{equation*}
	  j^2_{\frac{n}{2}-1,1}(\omega_n\, \AVR)^{\frac{2}{n}} \leq \lambda_1(\Omega)\abs{\Omega}^{\frac2n},
	\end{equation*}
	where $j_{\frac{n}{2}-1,1}$ denotes the first positive zero of  Bessel function of the first kind with order $\frac{n}{2}-1$.
\end{rem}

From the Faber-Krahn inequality \eqref{ineq:FK}, we can deduce  estimates for the second eigenvalue of Laplacian.

\begin{cor}\label{cor:HKS}
Under the same assumptions as in Corollary \eqref{cor:FK}, 
$\lambda_1(\Omega)$ denotes the second eigenvalue of  Dirichlet Laplacian
\begin{equation}\label{Equ:DL2}
	\left\{
	\begin{array}{ll}
		-\Delta_g  u= \lambda_2(\Omega) \, u,  & \quad\text{in} \quad \Omega,\\
		u =0, & \quad\text{on} \quad\partial\Omega.
	\end{array}
	\right.
\end{equation}
Then we have
\begin{equation}\label{ineq:HKS}
	\lambda_2(\Omega)> 2^\frac{2}{n}j^2_{\frac{n}{2}-1,1}\left(\frac{\omega_n\, \AVR}{\abs{\Omega}}\right)^{\frac{2}{n}}.
\end{equation}
\end{cor}
\begin{rem}
	When  $M^n$ is Euclidean space $\R^n$,  the result is known as  Krahn-Hong-Szeg\"{o} inequality.
\end{rem}

As another application of Theorem \ref{th_main_f}, we will establish the comparison theorem  for so called $L^1$ and $L^\infty$ moment spectrum  between complete Riemannian manifold and Euclidean space. Let $(M, g)$ be a complete Riemannian manifold and  $\Omega\subset M$ be a smoothly bounded domain. 
Assume that   $u$ solves the equation 
\begin{equation}\label{Equ:Torsion}
\left\{
\begin{array}{ll}
-\Delta_g u= 1,  & \text{in}\, \Omega,\\
u =0, & \text{on}\, \partial\Omega.
\end{array}
\right.
\end{equation}
The  torsional rigidity $T(\Omega)$  of  $\Omega$ is defined by
\begin{equation}
T(\Omega) = \int_{\Omega} u(x)\, d\mu_g(x).
\end{equation}
For the background on torsional rigidity, one can refer to \cite{Polya48,Baernstein19, Kesavan2006}.
In general,  let $u_k$ be  solution of a hierarchy of Poisson equation
\begin{equation}\label{Equ:moment spec}
\left\{\begin{aligned}
-\Delta_g u_k=& ku_{k-1} \quad \text{in }\Omega,\\
 u_k=&0 \quad \text{on }\partial \Omega,\qquad k=1,\cdots,
\end{aligned}\right.
\end{equation}
where  $u_0=1$ by convention. For a positive integer $k$, we define
\begin{equation}\label{L1-momnet-spec}
T_k(\Omega)=\int_{\Omega}u_kd\mu_g,\qquad \text{for}\quad k=1,\cdots,
\end{equation}
and   
\begin{equation}\label{Equ:L-infty-moment spec}
J_k(\Omega)=\sup_{x\in\Omega} u_{k},\qquad \text{for}\quad k=1,\cdots.
\end{equation}
The collection $\{T_k(\Omega) \}_{k=1}^{\infty}$ and $\{J_k(\Omega)\}_{k=1}^{\infty}$ are called the $L^1$-moment spectrum and  $L^\infty$-moment spectrum of $\Omega$, respectively.
For interpretation for $L^1$ and $L^\infty$-moment spectrum in probability theory, we refer the reader to \cite{CGL15,MD02,MD13}. For related results for  
torsional rigidity, $L^1$ and $L^\infty$-moment spectrum, one can consult \cite{AS,BS,CGL15,HChen,ColladayLM18,GHM2015,HMP16,KMM,MD02,MD13} and references therein.
\begin{cor}\label{cor:torsion}Let $(M,g)$ be a  noncompact, complete $n$-dimensional Riemannian manifold with  nonnegative Ricci curvature and  $\AVR>0$. Assume that $\Omega$ is a bounded domain in $M$ and $\Omega^\sharp$ is an Euclidean ball satisfying $\AVR \abs{\Omega^{\sharp}}=\abs{\Omega}$ .  For the  $L^1$-moment spectrum of $\Omega$, we have
	\begin{equation}\label{inequ:Saint-Vent}
	T_k(\Omega) \leq \AVR \, T_k(\Omega^\sharp),\quad k=1,\cdots.
	\end{equation}
In particular, for $k=1$, the Saint-Venant inequality holds
\begin{equation}\label{Inq:Saint-Venant}
T(\Omega) \leq \AVR \, T(\Omega^\sharp).
\end{equation}	
For $L^\infty$-moment spectrum of $\Omega$, we have 
\begin{equation}\label{inequ:L-infty}
J_k(\Omega)\leq J_k(\Omega^\sharp).
\end{equation}
Moreover, equality  holds for any $k$ in \eqref{Inq:Saint-Venant} or \eqref{inequ:L-infty} if and only if  $(M,g)$ is isometric to Euclidean space $(\R^n,g_0)$ and $\Omega$ is isometric to Euclidean ball $\Omega^\sharp$.
\end{cor}
\begin{rem}The inequality \eqref{Equ:L-infty-moment spec} can also be written by
		\begin{equation}\label{inequ:Saint-Ventk}
		\begin{aligned}
		T_k(\Omega)\abs{\Omega}^{-\frac{n+2k}{n}} \leq &\left( \AVR \omega_n\right)^{-\frac{2k}{n}} \, T_k(B_1)\omega_n^{-1}, \,  \quad k=1,\cdots,
		\end{aligned}
		\end{equation}
	where $B_1$ is the unit ball in Euclidean space $\mathbb{R}^n$.
	
The quantity $T_k(B_1)\omega_n^{-1}$ can be computed explicitly,	for instant, $k=1,2$,
	\begin{equation*}
		\frac{T_1(B_1)}{\omega_n}=\frac{1}{n(n+2)},\qquad \frac{T_2(B_1)}{\omega_n}=\frac{4}{n^2(n+2)(n+4)}. 
	\end{equation*}
In particular, the Saint-Venant inequality \eqref{Inq:Saint-Venant} can also be written as
		\begin{equation}\label{inequ:Saint-Vent1}
		\left(\AVR \omega_n\right)^{\frac 2n}T(\Omega)\abs{\Omega}^{-\frac{n+2}{n}}\leq \frac{1}{n(n+2)}.
		\end{equation}
\end{rem}

Our another aim here is to obtain a reverse H\"older inequality for eigenfunctions of Dirichlet eigenvalue problem. In 1972, Payne and Rayner \cite{PR72} proved that the eigenfunction $u$ of the Dirichlet Laplacian 
corresponding to the first eigenvalue $\lambda_1(\Omega)$ for a bounded planar domain $\Omega\subset \R^2$ 
satisfies a reverse H\"older inequality
\begin{equation*}
\frac{\norm{u}_{L^2(\Omega)}}{\norm{u}_{L^1(\Omega)}} \leq\frac{\sqrt{\lambda_1(\Omega)}}{2\sqrt{\pi}}  .
\end{equation*}
The equality occurs  if and only if $\Omega$ is a disk. In 1981,
Kohler-Jobin \cite{KJ81} obtained an isoperimetric comparison  between  the $L^2$ and  $L^1$ norms of the eigenfunction for Dirichlet Laplacian for bounded domain in $\R^n (n\geq 3)$. In 1982, Chiti \cite{Chiti82} 
 generalized   the reverse H\"older inequality for the norms $L_{q}$
 and $L_{p},$  $q\geq p >0$, for bounded domains of $\mathbb{R}^{n} (n\geq2).$ Chiti's comparison results was generalized to bounded domain in hemisphere by Ashbaugh and Benguria \cite{AB01}  and in hyperbolic space by Benguria and Linde \cite{BL07}, respectively. Using isoperimetric comparison the result in \cite{Chiti82}   was extended to compact Riemannian manifolds whose Ricci curvature has positive lower bound and the integral Ricci curvature condition \cite{HChen}. In this paper, we extend Chiti's result to complete Riemannian manifolds with nonnegative Ricci curvature.
\begin{thm}[]{}\label{thm:Chiti}
Let $(M,g)$ be a  noncompact, complete $n$-dimensional Riemannian manifold with  nonnegative Ricci curvature and  $\AVR>0$. Let $\Omega$ be a bounded domain in $(M,g)$ and $u$  be one of  solution	 corresponding to Dirichlet eigenvalue problem
\begin{equation}\label{Equ:DLM}
\begin{cases}
-\Delta_g  u= \lambda \, u,  & \mbox{in $\Omega$},\\
u =0, & \text{on} \, \partial\Omega.
\end{cases}
\end{equation}
For real numbers  $p$ and  $q$ satisfying $q \geq p > 0$, then we have
		\begin{equation}\label{Chiti0}
		\frac{\norm{u}_{L^q(\Omega)}}{\norm{u}_{L^p(\Omega)}} \leq K\left(p,q,\lambda,n,\AVR\right),
		\end{equation}
		where 
		\begin{equation}\label{}
	K\left(p,q,\lambda,n,\AVR\right)=
	\left(n\omega_{n}\AVR\left(j_{\frac n2-1,1}\ld^{-\frac12} \right)^n\right)^{\frac{1}{q}-\frac{1}{p}}	\displaystyle\frac{\left(\int_{0}^{1}r^{n-1+q(1-\frac{n}{2})}J_{\frac{n}{2}-1}^{q}(j_{\frac{n}{2}-1,1}\,r)dr\right)^{\frac{1}{q}}
	}{\left(\int_{0}^{1}r^{n-1+p(1-\frac{n}{2})}J_{\frac{n}{2}-1}^{p}(j_{\frac{n}{2}-1,1}\,r)dr\right)^{\frac{1}{p}}}.
		\end{equation}
Furthermore, equality  holds in \eqref{Chiti0}  if and only if $(M,g)$ is isometric to Euclidean space $(\R^n,g_0)$,  $\Omega$ is isometric to Euclidean ball  with radius $j_{\frac n2-1,1} \ld^{-1/2}$ and $\ld$ is the first eigenvalue of Dirichlet eigenvalue problem \eqref{Equ:DLM}.
\end{thm}

The paper is organized as follows. In Section \ref{Prelim}, we recall the isoperimetric inequality of complete Riemannian manifold with nonnegative Ricci curvature in \cite{Brendle,AFM},  Schwarz rearrangement and its properties for measurable functions; 
 In Section \ref{pf:Talenti}, we establish Talenti's comparison theorem \ref{th_main_f} by using the isoperimetric inequality and Schwarz rearrangement;
 In Section \ref{pf:FK}, we prove Faber-Krahn inequality for the first eigenvalue of Dirichlet Laplacian and comparison results for $L^1$ and $L^\infty$-moment spectrum.  In the last section \ref{pf:Chiti}, we  establish Chiti's comparison theorem \ref{thm:Chiti} for eigenfunctions of Dirichlet Laplacian.

\section{Preliminaries}\label{Prelim}

\subsection{Isoperimetric inequalities}

 Let $(M,g)$ be a  noncompact, complete $n$-dimensional Riemannian manifold with  nonnegative Ricci curvature and Euclidean volume growth, which means  the asymptotic volume ratio positive, i.e.
 $$
 {\sf AVR}(g)=\lim_{r\to \infty}\frac{\abs{B_x(r)}_g}{\omega_n r^n}>0,
 $$ 
where $B_x(r)$ stands for open metric ball centered at $x\in M$ with radius $r>0$ and $\omega_n$ is denoted by  the volume of the unit ball in $\R^n$.  According to Brendle \cite[Corollary 1.3]{Brendle},  for every bounded domain $\Omega\subset M$ with smooth boundary,   the isoperimetric inequality holds
\begin{equation}\label{isoperi-Brendle}
\abs{\partial \Omega}\geq n\omega_n^\frac{1}{n} \ {\sf AVR}(g)^\frac{1}{n}\, \abs{\Omega}^\frac{n-1}{n}.
\end{equation}
The equality  holds in \eqref{isoperi-Brendle} if 
and only if $(M,g)$ is isometric to $(\mathbb R^n,g_0)$ and $\Omega$ is isometric to an Euclidean ball.   
Letting $\Omega^\sharp\subset \mathbb R^n$ be an Euclidean ball centered at origin satisfying $\AVR\abs{\Omega^\sharp}=\abs{\Omega}$,
the inequality \eqref{isoperi-Brendle} can be equivalently rewritten as 
\begin{equation}\label{isoperi-Brendle-2}
	\abs{\partial \Omega}\geq \AVR\ \abs{\partial \Omega^\sharp}.
\end{equation}
We also notice that the inequality \eqref{isoperi-Brendle}  is proved by Agostiniani et al. \cite[Theorem 1.8]{AFM} for $n=3$ and then extended  to $3\leq n\leq 7$ by Fogagnolo and Mazzieri \cite{FM202012}. The equality case in \eqref{isoperi-Brendle} is also characterized and the isoperimetric inequality still holds in ${\sf CD} (0,N)$  metric measure spaces based on the method of optimal mass transport by Balogh and Krist\'{a}ly in \cite{BK202012} .

\subsection{Schwarz rearrangement}
\begin{defn}
	Let $h: \Omega \to \R$ be a measurable function. The distribution function of $h$ is the function $\mu : [0,+\infty)\, \to [0, +\infty)$ defined by
	$$
	\mu_h(t)= \abs{\left\{x \in \Omega \, :\,  \abs{h(x)} > t\right\}}.
	$$
\end{defn}
Here, and in the whole paper, $\abs{A}$ stands for the $n$-dimensional  measure of the set $A$.

\begin{defn}	Let $\Omega$ be a bounded domain in  complete manifold $(M,g)$  and   $h: \Omega \to \R$ be a measurable function. The decreasing rearrangement $h^{\ast}:[0,\abs{\Omega}_g]\to \mathbb{R}$  is defined by using the distribution function,
	\begin{equation}
	h^{\ast}(s)=
	\begin{cases}
	\underset{\Omega}{\textup{ess sup}}\,h & \textup{if }\,s=0,\\
	\inf \{t:\mu_h(t)\leq s \} & \textup{if }\, t>0.
	\end{cases}
	\end{equation}
	The Schwarz rearrangement $h^\sharp : \Omega^{\ast} \to \mathbb{R}$  of $h$ is  defined by
	\begin{equation}\label{rearr-Sch}
	h^\sharp (x)= h^*(\AVR \abs{B(r)}_{\gz} ),\qquad \text{for}\quad x\in \Omega^{\sharp},
	\end{equation}
where $B(r)$ denotes the Euclidean  ball with radius $r$ centered at origin in $\R^n$.	
\end{defn}
The Schwarz rearrangement $h^\sharp$ and $h$ satisfies that
\begin{equation}\label{Eqn:RearrMean}
\mu_{h}(t)=\AVR\,\mu_{h^{\sharp}}(t).
\end{equation}
 It is easily checked that $h$, $h^*$ and $h^\sharp$ are equi-distributed in the sense that
\begin{equation}\label{Eqn:equi}
\displaystyle{\norm{h}_{L^p(\Omega)}=\norm{h^*}_{L^p(0, \abs{\Omega})}=\AVR^{\frac{1}{p}}\,\norm{h^\sharp}_{L^p(\Omega^\sharp)}}.
\end{equation}
An important property of the decreasing rearrangement is the Hardy-Littlewood inequality
\begin{equation}\label{inq:HL}
\int_{\Omega} \abs{h(x)w(x)} \, dx \le \int_{0}^{\abs{\Omega}} h^*(s) w^*(s) \, ds,
\end{equation}
where $h,w$ are measurable functions defined on $\Omega$.
By  choosing $w=\chi_{\left\lbrace\abs{u}>t\right\rbrace}$ in \eqref{inq:HL}, one has
\begin{equation}\label{inq:HL1}
\int_{\{x\in \Omega:\abs{u}>t\}} \abs{h(x)} \, dx \le \int_{0}^{\mu(t)} h^*(s) \, ds.
\end{equation}
\begin{lem} \cite{HLP}\label{lem:HLP}
	Let  $R, p, q$  be real numbers such that $0<p\leq q$,  $R>0$, and  $f,$ $g$  real
	functions in $L^{q}([0,R])$. If the decreasing rearrangements of $f$
	and $g$ satisfy the following inequality
	$$\int_{0}^{s} \left(f^{\ast}\right)^{p}dt\leq \int_{0}^{s}\left(g^{\ast}\right)^{p}dt, \quad \text{for all } \; s\in [0,R], $$  \\
	then
	$$\int_{0}^{R}f^{q}dt\leq \int_{0}^{R}g^{q}dt.$$
\end{lem}

\section{Proof of Theorem \ref{th_main_f}} \label{pf:Talenti}
In this section, we shall prove the Talenti's comparison theorem \ref{th_main_f} by using Schwarz rearrangement and  the isoperimetric inequality in \cite{Brendle}.

\begin{proof}
	Assuming that  $v$ is the solution of \eqref{Equ:Dirichlet2}, we claim that $v^\sharp=v$. Since $f^\sharp$ is radial function and $\Omega^\sharp$ is the Euclidean ball, $v$ solves the ordinary differential equation
	\begin{equation*}
\begin{cases}
	-\frac{1}{r^{n-1}}\left(r^{n-1}v^\prime\right)^\prime=f^\sharp,\\
	v^\prime(0)=0=v(R),
\end{cases}
	\end{equation*}
where $R$ is the radius of the ball $\Omega^\sharp$. Integrating by parts we have
\begin{equation}\label{v1}
v(r)=\gamma_n^{-2}\int_{\AVR \omega_nr^n}^{\abs{\Omega}}\xi^{-2+\frac2n} F(\xi)d\xi,
\end{equation}
where $\gamma_n=n\left(\AVR \omega_n \right)^\frac{1}{n}$ and  
$$
F(\xi)=\int_0^\xi f^*(\eta)d\eta.$$

From \eqref{v1}, it follows that $ v(r)$ is nonincreasing since $f$ is nonnegative , and so $v=v^{\sharp}$. Therefore, we have
\begin{equation}\label{v2}
v^\ast(s)=\gamma_n^{-2}\int_{s}^{\abs{\Omega}}\xi^{-2+\frac2n} F(\xi)d\xi.
\end{equation}
Let $u$ be the solution to \eqref{Equ:Dirichlet1}.  For $\displaystyle t\ge 0$, we denote by 
\begin{equation*}
\Omega_t=\{x \in \Omega: u(x)>t\}, \qquad \Gamma_t=\{x\in \Omega: u(x)=t\},
\end{equation*}
and by $\mu_u(t) = |\Omega_t|,$ the  volume of $\Omega_t$  in $(M^n,g)$.  The Sard's theorem 
implies that 
\begin{equation*}
\partial\Omega_t=\Gamma_t,
\end{equation*}
for almost every $t$.

A function  $u \in H^1_0(\Omega)$ is a weak solution to \eqref{Equ:Dirichlet1} if
\begin{equation}
	\label{Equ:weaksol}
	\int_{\Omega} \n_g u \cdot \n_g \phi \,d\mu_g = \int_{\Omega} f \phi \,d\mu_g, \qquad \forall\, \phi \in H^1_0(\Omega).
\end{equation}
For $t>0$ and  $h>0$, define 
\begin{equation}\label{fuct:test}
	\varphi_h(x)=\begin{cases}
			0, & \text{if}\quad 0<u<t,\\
				\frac{u-t}{h},  &\text{if}\quad t<u<t+h,\\
		1, &\text{if}\quad t+h<u. 	
	\end{cases}
\end{equation}
Putting the test function \eqref{fuct:test}  in \eqref{Equ:weaksol} and letting $h$ go to $0$ yields
\begin{equation}\label{Equ:weaksol1}
-\frac{d}{dt}\int_{\Omega_t} \abs{\n_g u }^2 \,d\mu_g = \int_{\Omega_t} f \,d\mu_g, 
\end{equation}
From the co-area formula, \eqref{Equ:weaksol1} and \eqref{inq:HL1}, we have
\begin{equation}\label{inq:div}
\begin{aligned}
\int_{\Gamma_t}\abs{\n_g u} d\sigma
=\int_{\Omega_t}fd\mu_g
\leq &\int_0^{\mu_u(t)}f^*(\eta)d\eta.
\end{aligned}
\end{equation} 
From the isoperimetric inequality \eqref{isoperi-Brendle}, the Cauchy-Schwarz inequality and \eqref{inq:div}, we infer
\begin{equation*}
\begin{aligned}
\gamma_n^2 \,\mu_u(t)^{2-\frac{2}{n}}\leq &\abs{\Gamma_t}^2\\
 \leq& \int_{\Gamma_t} \abs{\n_g u}d\sigma \int_{\Gamma_t} \frac{1}{\abs{\n_g u}}d\sigma\\
\leq & \left(-\mu_u^\prime(t)\right)\int_0^{\mu_u(t)}f^*(\eta)ds\\
=&\left(-\mu_u^\prime(t)\right) F(\mu_u(t)),
\end{aligned}
\end{equation*}
i.e.
\begin{equation}\label{inq-u1}
1\leq \gamma_n^{-2}\mu_u(t)^{-2+\frac2n}\left(-\mu_u^\prime(t)\right) F(\mu_u(t))
\end{equation}
Integrating both sides in \eqref{inq-u1} from $0$ to $t$ yields
\begin{equation}\label{ineq-u2}
	t\leq \gamma_n^{-2} \int_{\mu_u(t)}^{\abs{\Omega}}\xi^{-2+\frac2n}F(\xi)d\xi.
\end{equation}
From the definition of $u^\ast$ and the right continuous property for decreasing rearrangement, we can infer
\begin{equation*}
\begin{aligned}
u^\ast(s)\leq &\gamma_n^{-2} \int_{s}^{\abs{\Omega}}\xi^{-2+\frac2n} F(\xi)d\xi\\
=& v^\ast(s),
\end{aligned}
\end{equation*}
which gives the proof of  \eqref{compare:uv}.

For the case of equality in \eqref{compare:uv}, we adapt the technique of Kesavan\cite{Kesavan88}. If $u^\sharp=v$, we have
\begin{equation}\label{equiv:distr}
\mu_u(t)=\AVR\,  \mu_v(t),\qquad \text{for all } \quad t\geq 0.
\end{equation}
From the explicit expression \eqref{v1} for $v$, we get
\begin{equation*}
t=\gamma_n^{-2}\int_{\AVR \mu_v(t)}^{\abs{\Omega}}\xi^{-2+\frac2n} F(\xi)d\xi.
\end{equation*}
Differentiating with respect to $t$, we have
\begin{equation}\label{v3}
1=\gamma_n^{-2}\left(\AVR \mu_v(t)\right)^{-2+\frac2n} F(\AVR \mu_v(t))\left(-\AVR \mu_v^\prime(t)\right).
\end{equation} 
From \eqref{inq-u1}, \eqref{equiv:distr} and \eqref{v3}, we deduce
\begin{equation}\label{}
\begin{aligned}
\gamma_n^2 \,\mu_u(t)^{2-\frac{2}{n}}\leq &\abs{\Gamma_t}^2
=\left(-\mu_u^\prime(t)\right) F(\mu_u(t))
=\gamma_n^2\mu_u(t)^{2-\frac{2}{n}},
\end{aligned}
\end{equation}
which implies that 
\begin{equation*}
\abs{\Gamma_t}= \gamma_n \mu_u(t)^{1-\frac{1}{n}}.
\end{equation*}
  The equality appears in isoperimetric equality \eqref{isoperi-Brendle}   if 
and only if $\AVR=1$, $(M,g)$ is isometric to $(\mathbb R^n,g_0)$ and $\Omega_t$ is isometric to an Euclidean ball centered at origin for almost every $t\geq 0$.   Let $t_n$ denote a strictly decreasing sequence satisfying $t_n\to 0$ as $n \to \infty$. Then
\begin{equation*}
\Omega=\{x\in \R^n: u(x)>0\}=\bigcup_{n=1}^n \Omega_{t_n},
\end{equation*}
which implies that $\Omega$  is a nested union of Euclidean ball centered at origin in $\R^n$, and then we deduce that $\Omega$ is exactly an Euclidean ball $\Omega^\sharp$. This completes the proof of Theorem \ref{th_main_f}.
\end{proof}

\section{ Faber-Krahn inequality, comparisons for $L^1$ and $L^\infty$ moment spectrum} \label{pf:FK}

In this section, as applications of Theorem \ref{th_main_f} we will prove that Faber-Krahn inequality, comparison results for $L^1$ and $L^\infty$-moment spectrum. Following an idea contained in \cite{Kesavan88}, we can prove the  Faber-Krahn inequality for the first eigenvalue.

\begin{proof}[Proof of Corollary \ref{cor:FK} ]
Assume that $v$ solve the equation
\begin{equation}\label{comp:Kesavan1}
\begin{cases}
-\Delta_{g_0} v=\lambda_1(\Omega)\,u^\sharp,\qquad \text{in} \quad\Omega^\sharp,\\
v|_{\partial\Omega^\sharp}=0,
\end{cases}
\end{equation}
where $u^\sharp$ is the Schwarz rearrangement of the first eigenfunction $u$ corresponding to equation \eqref{Equ:DL}. From Talenti's comparison result in Theorem \ref{th_main_f}, we have 
\begin{equation*}
u^\sharp(x)\leq v(x),\qquad x\in \Omega^\sharp,
\end{equation*}
which implies that 
\begin{equation*}
\int_{\Omega^\sharp}vu^\sharp d\mu_{g_0}\leq \int_{\Omega^\sharp} v^2d\mu_{g_0}.
\end{equation*}
Multiplying $v$ both sides in \eqref{comp:Kesavan1} and integrating by parts, we get
\begin{equation*}
\ld_1(\Omega)=\frac{\int_{\Omega^\sharp}\abs{\n_{g_0} v}^2 d\mu_{g_0}}{\int_{\Omega^\sharp}v u^\sharp d\mu_{g_0}}\geq \frac{\int_{\Omega^\sharp}\abs{\n_{g_0} v}^2 d\mu_{g_0}}{\int_{\Omega^\sharp}v^2 d\mu_{g_0}}\geq \ld_1(\Omega^\sharp),
\end{equation*}
where we use the Rayleigh-quotient characterization of the first eigenvalue of $\Omega^\sharp$. This gives the proof of \eqref{ineq:FK}. If the equality holds in \eqref{ineq:FK},  all above inequalities become equalities, especially, $u^\sharp=v$. Hence the equality case follows directly from Theorem \ref{th_main_f}.
\end{proof}

\begin{proof}[Proof of Corollary \ref{cor:HKS}]
Assume that $u_2$ is the second eigenfunction associated to $\ld_2(\Omega)$. According to Courant’s nodal theorem, $u_2$ must change its sign. Defining
$$
\Omega_{+}=\left\{x \in \Omega: u_{2}(x)>0\right\}, \quad \Omega_{-}=\left\{x \in \Omega: u_{2}(x)<0\right\},
$$
we have
$$
\left\{\begin{array} { r l r l } 
	{ - \Delta u _ { 2 } } & { = \lambda _ { 2 } ( \Omega ) u _ { 2 } } & { } & { \text { in } \Omega _ { + } , } \\
	{ u _ { 2 } } & { = 0 } & { } & { \text { on } \partial \Omega _ { + } , }
\end{array} \quad \left\{\begin{array}{rlrl}
	-\Delta u_{2} & =\lambda_{2}(\Omega) u_{2} & & \text { in } \Omega_{-} \\
	u_{2} & =0 & & \text { on } \partial \Omega_{-} .
\end{array}\right.\right.
$$
Let $\mathbb{B}_{+}$ and $\mathbb{B}_{-}$ be two disjoint Euclidean balls with $\AVR\left|\mathbb{B}_{+}\right|=\AVR\left|\mathbb{B}_{-}\right|=|\Omega| / 2$. From the Faber-Krahn inequality \eqref{ineq:FK}, we deduce that
\begin{equation*}
	\ld_2(\Omega)\geq  \ld_1(\Omega{\pm}) \geq  j^2_{\frac{n}{2}-1,1}\left( \frac{\omega_n\, \AVR}{\abs{\Omega_{\pm}}}\right)^{\frac{2}{n}},
\end{equation*}
which implies that
\begin{equation}\label{ineq:HKS1}
	2\ld_2(\Omega)\geq j^2_{\frac{n}{2}-1,1}\left(\omega_n\, \AVR\right)^{\frac{2}{n}}\left(\frac{1}{\abs{\Omega_{+}}^{\frac{2}{n}}}+ \frac{1}{\abs{\Omega_{-}}^{\frac{2}{n}}}\right)\geq j^2_{\frac{n}{2}-1,1}\left(\omega_n\, \AVR\right)^{\frac{2}{n}}\frac{1}{\abs{\Omega}^{\frac{2}{n}}},
\end{equation}
where we use the convexity of $t^{\frac{2}{n}}$ and $\abs{\Omega_{+}}+\abs{\Omega_{-}}\leq \abs{\Omega}$ in the second inequality. If the equality holds in  \eqref{ineq:HKS1}, then $\Omega_{+}$ and $\Omega_{-}$  must be Euclidean balls  by Faber-Krahn inequality \eqref{ineq:FK}. Since $\Omega_{+}$ and $\Omega_{-}$ have the same volume $\frac12 \abs{\Omega}$, $\Omega$ would not be connected. This completes the proof of Corollary \ref{cor:HKS}.
\end{proof}

\begin{proof}[Proof of Corollary \ref{cor:torsion}]
	Let $\Omega^\sharp$ be the rearrangement of $\Omega$ in $\R^n
$.  Assume that  $v_0=1$ and  $v_k$ satisfy
\begin{equation}\label{Equ:LinftyBall}
\left\{\begin{aligned}
-\Delta_{g_0}v_k=&kv_{k-1},\qquad \text{in}\quad \Omega^{\sharp},\\
v_k|_{\partial \Omega^{\sharp}}=&0,\qquad  k=1,\cdots,
\end{aligned}\right.
\end{equation}
and $w_k$ solve
\begin{equation}
\left\{\begin{aligned}
-\Delta_{g_0} w_k=&ku_{k-1}^\sharp,\qquad \text{in}\quad \Omega^{\sharp},\\
v_k|_{\partial \Omega^{\sharp}}=&0,\qquad  k=1,\cdots,
\end{aligned}\right.
\end{equation}
where $u_k$ are the solutions of the equation \eqref{Equ:L-infty-moment spec}.

For $k=1$, from the definition of $T_1(\Omega)$ and  Theorem \ref{th_main_f}, we have 
\begin{equation*}
T_1(\Omega)=\int_{\Omega}u_1d\mu_g=\AVR \int_{\Omega^\sharp} u_1^\sharp d\mu_{g_0}\leq \AVR \int_{\Omega^\sharp} v_1 d\mu_{g_0}=\AVR T(\Omega^\sharp).
\end{equation*}
For $k=2$, since $v_1\geq u_1^\sharp$, the comparison theorem of elliptic equation implies $v_2\geq w_2$.
By Theorem \ref{th_main_f},
 $$w_2\geq u_2^\sharp.$$ 
Therefore, we deduce
\begin{equation*}
\begin{aligned}
T_2(\Omega)=\int_{\Omega}u_2d\mu_g=&\AVR \int_{\Omega^\sharp} u_2^\sharp d\mu_{g_0}\\
\leq& \AVR \int_{\Omega^\sharp} w_2 d\mu_{g_0}\\
\leq& \AVR \int_{\Omega^\sharp} v_2 d\mu_{g_0}\\
=&\AVR T_2(\Omega^\sharp).
\end{aligned}\end{equation*}   
By induction for $k$, we can infer, for $ k\geq 2$,
\begin{equation}\label{comp:uvw}
u_k^\sharp\leq w_k\leq v_k, \qquad \text{in}\quad \Omega^{\sharp}.
\end{equation}  
Finally, we deduce
\begin{equation*}
	T_k(\Omega)=\AVR \int_{\Omega^\sharp} u_k^\sharp d\mu_{g_0}\leq \AVR \int_{\Omega^\sharp} w_k d\mu_{g_0}\leq \AVR \int_{\Omega^\sharp} v_k d\mu_{g_0}=\AVR T_k(\Omega^\sharp),
\end{equation*}    
which give the proof of \eqref{inequ:Saint-Vent}. If the equality occurs in \eqref{inequ:Saint-Vent},  the conclusion follows Theorem \ref{th_main_f}.

For $k=1$, by \eqref{compare:uv} we have
\begin{equation*}
J_1(\Omega)=\sup_{x\in\Omega} u_{1}=\sup_{x\in\Omega^\sharp} u^{\sharp}_{1}\leq \sup_{x\in\Omega^\sharp} v_{1}=J_1(\Omega^\sharp)=\frac{\abs{\Omega}}{2n\AVR \omega_n}.
\end{equation*}
For $k\geq 2$, by \eqref{comp:uvw} we can infer
\begin{equation*}
J_k(\Omega)=\sup_{x\in\Omega} u_{k}=\sup_{x\in\Omega^\sharp} u^{\sharp}_{k}\leq \sup_{x\in\Omega^\sharp} w_k\leq \sup_{x\in\Omega^\sharp} v_{k}=J_k(\Omega^\sharp).
\end{equation*}

In the following, we discuss  the equality case in \eqref{inequ:L-infty}.  If $J_k(\Omega)=J_k(\Omega^\sharp)$ for $k\geq 1$, denote 
$$
\mu_k(t)=\mu_{u_k}(t)=\abs{\{x\in\Omega: u_k(x)>t\}}
$$
 and 
 $$
 \nu_k(t)=\nu_{u_k}(t)=\abs{\{x\in\Omega^\sharp: v_k(x)>t\}},
 $$
  where $u_k$ and $v_k$ are the solution of \eqref{Equ:L-infty-moment spec} and \eqref{Equ:LinftyBall}, respectively. From \eqref{comp:uvw}, $u_k^\sharp\leq v_k$, we have
\begin{equation*}
\mu_k(t)\leq \AVR\nu_k(t).
\end{equation*}
Define 
\begin{equation*}
U_k(\eta)=\int_0^\eta u_{k}^*(s)ds,\qquad V_k(\eta)=\int_0^\eta v_{k}(s)ds,
\end{equation*}
and 
\begin{equation*}
\widetilde{H}_k(\eta)=\gamma_n^{-2}\int_{\eta}^{\abs{\Omega}}\xi^{-2+\frac2n} U_{k-1}(\xi)d\xi,\qquad \widetilde{G}_k(\eta)=\gamma_n^{-2}\int_{\eta}^{\abs{\Omega}}\xi^{-2+\frac2n} V_{k-1}(\xi)d\xi.
\end{equation*}
Similar to \eqref{v1}, 
\begin{equation*}
v_k(r)=\gamma_n^{-2}\int_{\AVR \omega_n r^n}^{\abs{\Omega}}\xi^{-2+\frac2n} V_{k-1}(\xi)d\xi.
\end{equation*}
Letting $v_k(r)=t$,
\begin{equation*}
t=\gamma_n^{-2}\int_{\AVR \nu_k(t)}^{\abs{\Omega}}\xi^{-2+\frac2n} V_{k-1}(\xi)d\xi
\end{equation*}
Taking the derivative with respect to $t$ both sides, we have
\begin{equation*}
1=\gamma_n^{-2}\left(\AVR \nu_k(t)\right)^{-2+\frac2n} V_{k-1}(\AVR \nu_k(t)) \left(-\AVR \nu_k^\prime(t) \right).
\end{equation*}
Now recalling the proof of Theorem \ref{th_main_f}, similar arguments  to \eqref{inq-u1}, we obtain
\begin{equation*}
1\leq \gamma_n^{-2}\left(\mu_k(t)\right)^{-2+\frac2n} U_{k-1}(\mu_k(t)) \left(- \mu_k^\prime(t) \right),
\end{equation*}
which gives 
\begin{equation*}
\frac{d}{dt}\widetilde{G}_k(\AVR \nu_k(t))\leq \frac{d}{dt} \widetilde{H}_k(\mu_k(t)).
\end{equation*}
For $t=0$, 
\begin{equation*}
\begin{aligned}
\widetilde{G}_k(\AVR \nu_k(0))=&\widetilde{G}_k(\AVR \abs{\Omega^\sharp})=\widetilde{G}_k(\abs{\Omega})=0,\\
\widetilde{H}_k(\mu_k(0))=&\widetilde{H}_k(\abs{\Omega})=0,
\end{aligned}
\end{equation*}

If $J_k(\Omega)=J_k(\Omega^\sharp)$, we have
\begin{equation*}
\widetilde{G}_k(0)=\gamma_n^{-2}\int_{0}^{\abs{\Omega}}\xi^{-2+\frac2n} V_{k-1}(\xi)d\xi\geq \gamma_n^{-2}\int_{0}^{\abs{\Omega}}\xi^{-2+\frac2n} U_{k-1}(\xi)d\xi=\widetilde{H}_k(0).
\end{equation*}
Set
$$
\zeta(t)=\widetilde{G}_k(\AVR \nu_k(t))-\widetilde{H}_k(\mu_k(t)).
$$
Since
\begin{equation*}
\zeta(0)=0, \qquad  \zeta(J_k(\Omega))\geq 0,\qquad \zeta^\prime (t)\leq 0,
\end{equation*}
we get
 $$
 \zeta(t)\equiv0,
 $$
  which implies that
 $$\mu_k(t)=\AVR \nu_k(t),\qquad \text{for all}\quad t>0.$$
 This completes the proof of Corollary \ref{cor:torsion}.
\end{proof}

\section{Proof of Chiti's reverse H\"{o}lder inequality}\label{pf:Chiti}
In this section, we will prove Chiti's reverse H\"{o}lder inequality based on the isoperimetric inequality and Faber-Krahn inequality for the first eigenvalue.
\begin{lem}\label{comp-pointwise-chiti}
 Let $(M,g)$ be a  noncompact, complete $n$-dimensional Riemannian manifold with  nonnegative Ricci curvature and $\AVR>0$. Let $\Omega$ be a bounded domain in $(M,g)$ and $u$  be one of  solution to 
 \begin{equation}\label{Equ:DirichletEig1}
 \begin{cases}
 -\Delta_g u=\ld\, u,\qquad \text{in}\quad \Omega,\\
 u|_{\partial \Omega}= 0,\qquad \text{on}\quad \partial \Omega.
 \end{cases}
 \end{equation}
 Let $B_{\ld}$ be an Euclidean ball with the first eigenvalue  $\ld$ and $v$ be the solution to Dirichlet eigenvalue problem
\begin{equation}\label{Equ:DirichletEig2}
 \begin{cases}
-\Delta_{g_0} v=\ld\, v,\qquad \text{in}\quad B_{\ld},\\
v= 0,\qquad \text{on}\quad \partial B_{\ld}.
\end{cases}
\end{equation}
 If $\bar{u}:=\max_{x\in \Omega} u=v(0)$, then we have	
 	\begin{equation}\label{compare:chiti1}
 v^\ast(s)\leq u^\ast(s), \qquad s\in \left[0, \AVR \abs{B_{\ld}}\right].
 \end{equation}
Moreover, the equality  holds in \eqref{compare:chiti1} if and only if $(M,g)$ is isometric to Euclidean space $(\R^n,g_0)$ and $\Omega$ is isoperimetric to Euclidean ball $B_{\ld}$.
\end{lem}
\begin{proof} 
For any $s\in [0, \AVR \abs{B_\ld}]$, there exists a positive constant $c$ such that
\begin{equation*}
    v^\ast(s)\leq c u^*(s).
\end{equation*}	
Define 
\begin{equation*}
\underline{c}=\inf \left \{c|v^\ast(s)\leq c\, u^*(s), s\in \left[0, \AVR \abs{B_\ld}\right]\right\}.
\end{equation*}
Obviously, $\underline{c}\geq 1$ since $\bar{u}=z(0)$. From Faber-Krahn inequality in Corollary \ref{ineq:FK} and the domain monotonicity of the first eigenvalue,  we can infer $B_{\ld}\subseteq \Omega^\sharp$. If $B_{\ld}= \Omega^\sharp$, it is nothing to prove. Now assuming that $B_{\ld}\subset \Omega^\sharp$,  we get
$$
v^*(\abs{B_\ld})=0\quad \text{and} \quad  u^*(\abs{B_\ld})>0.
$$
Since $u^*(0)=\bar{u}=v(0)=v^*(0)$, there exists $s_0\in [0,\AVR\abs{B_\ld}]$ such that 
\begin{equation*}
 \underline{c} u^*(s_0)=v^*(s_0).
\end{equation*}
Setting
\begin{equation*}
w(s)=\begin{cases}
\underline{c} \, u^*(s), \qquad s\in [0,s_0],\\
v^*(s),\qquad s\in (s_1,\abs{B_\ld}],
\end{cases}
\end{equation*}
we have
\begin{equation}\label{compare:w}
-\frac{d}{ds} w(s)\leq \ld\, \gamma_n^{-2}\, s^{-2+\frac2n}\, \int_0^s w(\xi)d\xi.
\end{equation}
Let 
\begin{equation*}
w^\sharp(r)=w\left(\AVR\,\omega_{n} \, r^n\right),\qquad \text{ for}\quad r\in \left[0, j_{\frac n2-1}\ld^{-\frac12}\right].
\end{equation*}
By directly calculations and  \eqref{compare:w}, we have
\begin{equation*}
	\begin{aligned}
	\int_{B_\ld} \abs{\n_{g_0} w^\sharp}^2(r) d\mu_{g_0}=&\AVR^{-1}\int_{0}^{\AVR \abs{B_\ld}}\left(-\frac{d w(s)}{ds}\right)^2\gamma_n^2s^{2-\frac2n}ds\\
	\leq & \AVR^{-1} \ld \int_{0}^{\AVR \abs{B_\ld}}\left(-\frac{d w(s)}{ds}\right)\int_{0}^sw(\xi)d\xi ds\\
	=& \AVR^{-1} \ld \int_0^{\AVR \abs{B_\ld}}w^2(s)\, d\, s,
	\end{aligned}
\end{equation*}
and 
\begin{equation*}
\int_{B_\ld} \left(w^\sharp(r) \right)^2 d\mu_{g_0} =\AVR^{-1} \int_{0}^{\AVR \abs{B_\ld}} w^2(s)\, d\, s.
\end{equation*}
From the Rayleigh quotient on $B_\ld$,  we have
\begin{equation*}
\ld \leq \frac{\int_{B_\ld} \abs{\n_{g_0} w^\sharp}^2(r) d\mu_{g_0}}{\int_{B_\ld} \left(w^\sharp(r) \right)^2 d\mu_{g_0}}\leq \ld.
\end{equation*}
 By using the simplicity of the first eigenvalue and $u^\ast(0)=v(0)$, we get $\underline{c}=1$, which implies
 \begin{equation*}
 	v^\ast(s)\leq u^\ast(s), \qquad\text{for all}\, s\in [0,\AVR \abs{B_\ld}].
 \end{equation*}
 This completes the proof of Lemma \eqref{comp-pointwise-chiti}.
\end{proof}

\begin{proof}[Proof of Theorem \ref{thm:Chiti}]
 We	normalize $u$ such that 
 \begin{equation}\label{equ:norm}
 \norm{u}_{L^p(\Omega)}=\AVR^{\frac{1}{p}}\norm{v}_{L^p(B_\ld)},
 \end{equation}
 which is equivalent to 
 \begin{equation}\label{equ:normalized u}
 \int_{0}^{\abs{\Omega}}\left( u^{\ast}(s)\right)^p\, d\,s= \int_{0}^{\AVR \abs{B_\ld}}\left( v^{\ast}(s)\right)^p\, d\,s.
 \end{equation}
 From Lemma \ref{comp-pointwise-chiti}, we can infer that $v^\ast(0)\geq u^\ast(0)$. 
 
 Now we prove the theorem by dividing two cases.
 
 Suppose that $v^\ast(0)=u^\ast(0)$. Form Lemma \ref{comp-pointwise-chiti} and \eqref{equ:normalized u}, we have
 \begin{equation*}
 \int_{0}^{\abs{\Omega}}\left( u^{\ast}(s)\right)^p\, d\,s= \int_{0}^{\AVR \abs{B_\ld}}\left( v^{\ast}(s)\right)^p\, d\,s\leq  \int_{0}^{\AVR \abs{B_\ld}}\left( u^{\ast}(s)\right)^p\, d\,s.
 \end{equation*}
This  implies that $\abs{\Omega}_g=\AVR \abs{B_\ld}$. Furthermore,  $\Omega^\sharp=B_\ld$ and $u^\sharp=v$, which implies equality cases in \eqref{Chiti0}. 

Now suppose that $v^\ast(0)>u^\ast(0)$. In this case, we get
\begin{equation*}
u^\ast(\abs{\Omega}_g)=0,\qquad v^\ast(\AVR\abs{B_\ld})=0,\qquad \text{and}\quad   \abs{B_\ld}<\abs{\Omega^\sharp}.
\end{equation*}

We claim that there exists only one $s_1\in [0,\AVR \abs{B_\ld}]$ such that $v^\ast(s_1)=u^\ast(s_1)$.  In fact, if one can find $s_2\in (0,\AVR\abs{B_\ld})$ and $s_2>s_1$, such that 
\begin{equation}\label{equ:count}
	u^\ast(s_2)=v^\ast(s_2),\qquad u^\ast(s)>v^\ast(s),\qquad \text{for }\quad s\in (s_1,s_2).
\end{equation}
Defining
\begin{equation*}
h(s)=\begin{cases}
	v^\ast(s),\qquad s\in [0,s_1]\cup \left[s_2,\AVR \abs{B_\ld}\right],\\
	u^\ast(s),\qquad s\in (s_1,s_2),
\end{cases}
\end{equation*} 
 we have
 \begin{equation}\label{comp:interdiff-h}
 	-\frac{d }{ds} h(s)\leq \ld \,\gamma_n^{-2}\, s^{-2+\frac 2n}\,\int_0^s h(\xi)d\xi.
 \end{equation}
 Setting 
 \begin{equation*}
 	h^\sharp(r)=h^*(\AVR \omega_n r^n), \qquad r\in \left[0, j_{\frac n2-1}\ld^{-\frac12}\right],
 \end{equation*}
 by \eqref{equ:count} and \eqref{comp:interdiff-h}, we get 
 \begin{equation*}
 \frac{\int_{B_\ld}\abs{\n_{g_0} h^\sharp}^2 d\mu_{g_0}}{\int_{B_\ld} \left(h^\sharp\right)^2 d\mu_{g_0}}<\ld,
 \end{equation*}
which is a contradiction with the minimum property of the first eigenvalue $\ld$ of $\Omega^\sharp$. Therefore, there exists $s_1$ such that
\begin{equation}\label{comp:uv}
	u^\ast(s) \begin{cases}
		\leq& v^\ast(s), \qquad s\in [0,s_1];\\
		\geq & v^\ast(s), \qquad s\in (s_1,\AVR \abs{B_\ld}].
	\end{cases}
\end{equation}
The function $v^\ast(s)$ can be extended to $0$ as $s\in (\AVR \abs{B_\ld},\abs{\Omega}_g]$. We claim that 
\begin{equation}\label{comp:inter-uv}
\int_0^s\left(u^\ast\right)^p(\xi)d\xi \leq \int_0^s \left(u^\ast\right)^p(\xi) d\xi,\qquad s\in \left[0,\AVR \abs{B_\ld}\right].
\end{equation}
In fact, if $s\in [0,s_1]$,  \eqref{comp:inter-uv} holds obviously by \eqref{comp:uv}. For $s\in (s_1,\AVR \abs{B_\ld}]$,  we deduce, from \eqref{equ:normalized u} and \eqref{comp:uv},
\begin{equation*}
	\begin{aligned}
	\int_0^s\left(u^*\right)^p(\xi)\, d\,\xi =& \int_0^{\abs{\Omega}}\left(u^*\right)^p(\xi)\, d\,\xi-	\int_s^{\abs{\Omega}}\left(u^*\right)^p(\xi)\, d\,\xi\\
	=&\int_0^{\AVR\abs{B_\ld}}\left(v^*\right)^p(\xi)\, d\,\xi-	\int_s^{\abs{\Omega}}\left(u^*\right)^p(\xi)\, d\,\xi\\
	\leq & \int_0^{\AVR\abs{B_\ld}}\left(v^*\right)^p(\xi)\, d\,\xi-	\int_s^{\abs{\Omega}}\left(v^*\right)^p(\xi)\, d\,\xi\\
	=&\int_0^s \left(v^\ast(\xi) \right)^p d\xi.
	\end{aligned}
\end{equation*}
From Lemma \ref{lem:HLP}, we can infer
\begin{equation*}
	\int_{0}^{\abs{\Omega}}\left(u^\ast\right)^q(s)\, d\, s\leq \int_{0}^{\abs{\Omega}}\left(v^\ast\right)^q(s)\, d\, s=\int_0^{\AVR \abs{B_\ld}}\left(v^\ast\right)^q(s)\, d\,s,
\end{equation*}
which is equivalent to 
\begin{equation*}
	\norm{u}_{L^q(\Omega)}\leq \AVR^{\frac{1}{q}}\norm{v}_{L^q(B_\ld)}.
\end{equation*}
Now we assume that the equality holds in \eqref{Chiti0}, from \eqref{equ:norm}, one can infer that 
\begin{equation*}
\abs{\Omega}=\AVR \abs{B_{\lambda}}.
\end{equation*}
Since $B_\lambda$ and $\Omega^{\sharp}$ the ball center at origin, it yields 
\begin{equation*}
\Omega^\sharp=B_\lambda.
\end{equation*}
By assumption of Lemma \ref{comp-pointwise-chiti}, $\lambda$ is the first eigenvalue of $B_\lambda,$ hence of $\Omega^\sharp$. Therefore, we deduce that 
\begin{equation*}
\lambda_1(\Omega)=\lambda_1(\Omega^\sharp)=\lambda.
\end{equation*}
From Corollary \ref{cor:FK}, we infer that $(M,g)$ is isometric to Euclidean space $(\R^n,g_0)$,  $\Omega$ is isometric to Euclidean ball  with radius $j_{\frac n2-1} \ld^{-1/2}$ and $\ld$ is the first eigenvalue of problem \eqref{Equ:DLM}.
This completes the proof of Theorem \ref{thm:Chiti}.
 \end{proof}


 {\it Acknowledgement.}   The authors wish to thank Professor A. Krist\'{a}ly for some valuable discussions and
 helpful comments.

\providecommand{\bysame}{\leavevmode\hbox
	to3em{\hrulefill}\thinspace}

\vspace{1cm}

\begin{flushleft}
	Daguang Chen,
	E-mail: dgchen@tsinghua.edu.cn\\
Haizhong Li,
E-mail:	lihz@tsinghua.edu.cn\\
Department of Mathematical Sciences, Tsinghua University, Beijing, 100084, P.R. China 	
	
\end{flushleft}

\end{document}